\documentclass{amsart}

\usepackage[shortalphabetic]{amsrefs}
\usepackage{amsmath}
\usepackage{amssymb}
\usepackage{color}

\usepackage{verbatim}
\usepackage{enumerate}

\DeclareMathAlphabet{\mathpzc}{OT1}{pzc}{m}{it}

\numberwithin{equation}{section}
\newtheorem{thm}[equation]{Theorem}
\newtheorem*{thm*}{Theorem}
\newtheorem{lem}[equation]{Lemma}
\newtheorem{prop}[equation]{Proposition}

\newtheorem*{remark*}{Remark}

\newtheorem{remark}[equation]{Remark}

\newtheorem{coro}[equation]{Corollary}

\newlength{\categorywidth}       
\usepackage{calc}
\newlength{\infowidth}          % Width of info (right side)
\newlength{\categorysep}
\newlength{\entrysep}
\DeclareMathOperator{\GL}{GL}
\DeclareMathOperator{\SL}{SL}
\DeclareMathOperator{\Sp}{Sp}

\DeclareMathOperator{\Orth}{O}

\DeclareMathOperator{\PSL}{PSL}

\DeclareMathOperator{\Aut}{Aut}
\DeclareMathOperator{\isom}{Isom}

\DeclareMathOperator{\rad}{rad}

\newcommand{\la}{\langle}
\newcommand{\ra}{\rangle}
\renewcommand{\phi}{\varphi}
\renewcommand{\leq}{\leqslant}
\renewcommand{\geq}{\geqslant}
\newcommand{\F}{\Bbb{F}_{q}}

\newcommand{\Fp}{\Bbb{F}_p}
\newcommand{\Fk}{\Bbb{F}_{2^k}}

\begin{document}

\title{Orthogonal groups in characteristic 2 acting on polytopes of high rank}
\author{P. A. Brooksbank}
\address{
	Department of Mathematics\\
	Bucknell University\\
	Lewisburg, PA 17837\\
	United States
}
\email{pbrooksb@bucknell.edu}
\thanks{This work was partially supported by a grant from the Simons Foundation (\#281435 to Peter Brooksbank).
}

\author{J. T. Ferrara}
\address{
	Department of Mathematics\\
	Bucknell University\\
	Lewisburg, PA 17837\\
	United States
}
\email{jtf019@bucknell.edu}
\author{D. Leemans}
\address{
	D\'{e}partement de Math\'{e}matique\\
	C.P.216 -- Alg\`{e}bre et Combinatoire\\
	Boulevard du Triomphe\\
	1050 Bruxelles \\
	Belgium
}
\email{dleemans@ulb.ac.be}

\date{\today}
\keywords{quadratic form, orthogonal group, string C-group, abstract regular polytope}
\begin{abstract}
We show that for all integers $m\geq 2$, and all integers $k\geq 2$, 
the orthogonal groups $\Orth^{\pm}(2m,\Fk)$ act on abstract regular polytopes 
of rank $2m$, and the symplectic groups $\Sp(2m,\Fk)$ act on abstract regular polytopes of rank $2m+1$.
\end{abstract}

\maketitle

%%%%%%%
\section{introduction}
In a series of three papers, Monson and Schulte studied families of crystallographic 
Coxeter groups with string diagrams~\citelist{\cite{MS1}\cite{MS2}\cite{MS3}}. In those
works, the authors started with the standard real representations and considered reductions 
modulo odd primes, producing finite quotients of the string Coxeter groups
represented in orthogonal spaces over prime fields $\Bbb{F}_p$.
They showed that in many cases the images 
$\rho_0,\ldots,\rho_{d-1}$
of the generating reflections of the string Coxeter group
inherited a certain 
{\em intersection property} possessed by the parent group, 
namely
\begin{align}
\label{eq:IP}
\forall I,J\subseteq \{0,\ldots,d-1\} & & \la \rho_i\colon i\in I\ra
\cap \la \rho_j\colon j\in J\ra=\la \rho_k\colon k\in I\cap J\ra.
\end{align}
Any group generated by involutions satisfying condition~(\ref{eq:IP}) and the additional (string) commuting
condition 
\begin{align}
\label{eq:string}
\forall i,j\in\{0,\ldots,d-1\} & & \rho_i\rho_j\neq \rho_j\rho_i~~\Longleftrightarrow~~|i-j|=1
\end{align}
is called a {\em string C-group of rank $d$}. 
In~\cite[Section 3]{MS3} the authors study the ``3-infinity" string Coxeter groups, showing in~\cite[Theorem 3.1]{MS3} that 
reductions of such groups modulo odd primes are always string C-groups. From this result, one can deduce the following:

%%%
\begin{thm}
\label{thm:MS}
Let $p$ be an odd prime, and $V$ an $\Fp$-space of dimension $d\geq 3$.
There is a nondegenerate quadratic form $\phi$ on $V$ and 
$\isom(V_{\phi})'\leq H\leq \isom(V_{\phi})$, where $\isom(V_{\phi})$ is the 
isometry group of $\phi$, such that $H$ is a string C-group of rank $d$ generated by reflections.
\end{thm}

The modular reduction method employed by Monson and Schulte
also works in some cases over extensions of $\Fp$, but it doesn't apply at all to fields of characteristic 2.
One key difference in characteristic 2 is that reflections no longer exist; they are replaced by
natural analogues called {\em symmetries} (defined in Section~\ref{sec:symmetries}).
The main objective of the current paper is to prove the following supplement to Theorem~\ref{thm:MS}:

%%%
\begin{thm}
\label{thm:main-orth}
Let $k$ be a positive integer, $\Fk$ the field of $2^k$ elements, and $V$ an $\Fk$-space of dimension $d\geq 3$.
Then there is a quadratic form $\phi$ on $V$ such that $\isom(V_{\phi})$ is a string C-group of rank $d$
generated by symmetries if, and only if, $k\geq 2$. The radical of $V$ is 0- or 1-dimensional according as $\dim V$ is even or odd.
\end{thm}
%%%

Interest in string C-groups stems from their intimate connection to highly symmetric 
incidence structures called {\em abstract regular polytopes}. 
Indeed a fundamental result in~\cite[Section 2]{MS2002} shows that $\mathcal{P}$ is an abstract regular polytope if, and only if, 
$\Aut(\mathcal{P})$ is a string C-group---the
%, where the rank of the polytope coincides with the rank of the string C-group. 
studies of abstract regular polytopes and string C-groups are equivalent.
Refining the choice of $\phi$ in Theorem~\ref{thm:main-orth} we prove the following result concerning polytopes of high
rank associated to classical groups. 
%in characteristic 2.

%%%
\begin{coro}
\label{coro:main-sp}
For each integer $k\geq 2$, positive integer $m$, and $\epsilon\in\{-,+\}$,
the orthogonal group $\Orth^{\epsilon}(2m,\Fk)$ acts as the group of automorphisms
of an abstract regular polytope of rank $2m$, and the symplectic group $\Sp(2m,\Fk)$ acts as
the group of automorphisms of an abstract regular polytope of rank $2m+1$.
\end{coro}

Very few constructions of abstract regular polytopes of high rank are known.
Indeed, aside from the Monson-Schulte constructions and the ones presented here
to prove Corollary~\ref{coro:main-sp}, the only
existing constructions for arbitrary rank $r\geq 3$ are for the symmetric groups~\cite{FL, FLM1}
and the alternating groups~\cite{FLM}.

%%%%%
\section{symplectic and orthogonal spaces}
\label{sec:orth-space}
We now introduce the necessary background on classical groups and their geometries;
the reader is referred to Taylor's book~\cite{Taylor} for further information.
\smallskip

\begin{quotation}
{\em For the remainder of the paper, $V$ will denote a vector space of dimension $d\geq 3$ 
over the finite field $\F$ of $q=2^k$ elements.}
\end{quotation}
\smallskip

Let $(\,,\,)$ denote a symmetric bilinear form on $V$.
For $U\leq V$, the subspace 
\[
U^{\perp}=\{v\in V\colon (v,U)=0\}
\]
 is the {\em perp-space} of $U$,
$\rad(U)=U^{\perp}\cap U$ is the {\em radical} of $U$, and $U$ is {\em nondegenerate} if $\rad(U)=0$.
For $S\subseteq  \GL(V)$, the subspace 
\[
[V,S]=\la v-vs\colon v\in V,\,s\in S\ra
\]
 is the {\em support} of $S$.
Let $\phi\colon V\to \F$ be a quadratic form on $V$ whose associated symmetric form
is $(\,,\,)$ in the sense that 
\begin{align*}
\forall u,v\in V & &  (u,v)=\phi(u+v)+\phi(u)+\phi(v).
\end{align*}
When we wish to stress that we are regarding $V$ as an {\em orthogonal space}
equipped with $\phi$, rather than as a {\em symplectic space} equipped with $(\,,\,)$, 
we shall do so by writing $V_{\phi}$. For $U\leq V_{\phi}$, we write $U_{\phi}$ to denote the orthogonal
space obtained by restricting $\phi$ to $U$. A nonzero vector $u\in V$  is {\em singular} if $\phi(u)=0$ 
and {\em nonsingular} otherwise.  A subspace $U$ of $V$ is {\em totally singular} if $\phi(u)=0$ for all $u\in U$.

%%%
\subsection{Matrices} We will find it convenient later on to compute with matrices representing
symmetric and quadratic forms. Fix a basis $v_1,\ldots,v_d$ of $V$, and put $B:=[[(v_i,v_j)]]$.
Writing elements of $V$ as row vectors relative to our fixed basis, one evaluates the
bilinear form as 
\begin{align*}
\forall u,v\in V & &
(u,v)=u\,B\,v^{{\rm tr}}. 
\end{align*}
If $\Phi=[[\alpha_{ij}]]$, where $\alpha_{ij}=(v_i,v_j)$ if $i<j$, $\alpha_{ij}=0$ if $i>j$,
and $\alpha_{ii}=\phi(v_i)$, then 
\begin{align*}
\forall v\in V & &
\phi(v)=v\,\Phi\, v^{{\rm tr}}, 
\end{align*}
and $B=\Phi+\Phi^{{\rm tr}}$.
If $w_1,\ldots,w_d$ is another basis of $V$ and $C$ is the matrix whose rows are the vectors
representing the $w_i$ relative to $v_1,\ldots,v_d$, then $CBC^{{\rm tr}}=[[(w_i,w_j)]]$ is the matrix
representing $(\;,\;)$ relative to $w_1,\ldots,w_d$.

%%%
\subsection{Classification of lines}
We work with 2-dimensional subspaces of orthogonal spaces,
which we often think of as lines in the associated projective space. For a line $L$ in an orthogonal space
$V_{\phi}$, one of the following holds:
\begin{itemize}
\item $L$ is {\em asingular}: all nonzero vectors are nonsingular.
\item $L$ is {\em singular}: $L=\la e,b\ra$ with $\phi(e)=(e,b)=0$ and $\phi(b)\neq 0$.
\item $L$ is {\em hyperbolic}: $L=\la e,f\ra$ with $\phi(e)=\phi(f)=0$ and $(e,f)\neq 0$.
\item $L$ is {\em totally singular}: all nonzero vectors are singular.
\end{itemize}
The four possibilities are ordered according to how many singular points (1-spaces) each possesses: $0,1,2,q+1$, respectively.

%%%
\subsection{Isometries}
Associated with symplectic and orthogonal spaces are their groups of isometries:
\begin{align}
\label{eq:isom-bil}
\isom(V) & =  \{g\in\GL(V)\colon (ug,vg)=(u,v)~\mbox{for all}~u,v\in V\},~\mbox{and} \\
\label{eq:isom-quad}
\isom(V_{\phi}) & =  \{g\in\GL(V)\colon \phi(vg)=\phi(v)~\mbox{for all}~v\in V\}.
\end{align}
When $V$ is nondegenerate, $\dim V=2m$ is even and $\isom(V)$ is often called the 
{\em symplectic group} on $V$, denoted $\Sp(V)$.
In this case there are two isomorphism types of orthogonal group corresponding
to two isometry types of quadratic form. These are distinguished by the dimension of a
maximal totally singular subspace of $V$, which can be either $m$ or $m-1$. In the former case, the
associated orthogonal group is denoted $\Orth^+(V_{\phi})$, and in the latter case $\Orth^-(V_{\phi})$.

We will also need to work with odd-dimensional (degenerate) orthogonal spaces
$V_{\phi}$ when their associated symplectic spaces have 1-dimensional radicals. 
There are two cases to consider. 
First, if $V^{\perp}$ is a singular point, then $\psi(v+V^{\perp}):=\phi(v)$ is a well-defined nondegenerate
quadratic form on $V/V^{\perp}$, and 
\begin{equation}
\label{eq:sing-odd}
\isom(V_{\phi})\cong O_2(\isom(V_{\phi}))\rtimes[\F^{\times}\times\Orth^{\pm}(V/V^{\perp})],
\end{equation} 
where $O_2(\isom(V_{\phi}))$---the largest normal 2-subgroup of $\isom(V_{\phi})$---is an abelian group 
isomorphic to the
additive group of the $\F$-space $V/V^{\perp}$.
Hence, there are two isomorphism types for $\isom(V_{\phi})$, corresponding to the two isometry types of 
$V/V^{\perp}$. The second possibility is that
$V^{\perp}$ is a nonsingular point. Here, one can only induce the bilinear form on $V/V^{\perp}$, and
this leads to the well-known isomorphism 
\begin{equation}
\label{eq:nonsing-odd}
\isom(V_{\phi})\cong\Sp(V/V^{\perp}).
\end{equation}
{\em Henceforth, when we say $V_{\phi}$ is an orthogonal space, it is assumed that $\dim V^{\perp}\leq 1$}.

%%%%%
\section{symmetries of orthogonal spaces}
\label{sec:symmetries}
We shall not require a detailed description of the involution classes of the symplectic and orthogonal groups; 
we need just one special type.
If $x\in V$ with $\phi(x)\neq 0$, then 
\begin{align}
\label{eq:sigma}
\sigma_x\colon v\mapsto v+\frac{(v,x)}{\phi(x)}x & & (v\in V)
\end{align}
is an involutory isometry $V_{\phi}$ known as a {\em symmetry}. 
For $\alpha\in\F$, note $\sigma_{\alpha x}=\sigma_x$, so $\sigma_x$ is uniquely determined 
by the nonsingular 1-space $\la x\ra$. The group $\la \sigma_x\ra$ is the subgroup of 
$\isom(V_{\phi})$ that induces the identity on $\la x\ra^{\perp}$ 
and on $V/\la x\ra$. 

Symmetries have very nice properties that recommend them as generators for string C-groups.
First, symmetries $\sigma_x$ and $\sigma_{y}$ commute if, and only if, $(x,y)=0$.
 Secondly, the support of a symmetry is simply the nonsingular 1-space that defines it, namely
$[V,\la \sigma_x\ra]=\la x\ra$.
Finally, it is well known that $\isom(V_{\phi})$ is generated by symmetries when 
$V_{\phi}$ is {\em nonsingular}~\cite[Theorem~11.39]{Taylor}.
However, we shall need more detailed information about the groups generated by symmetries for our broader 
notion of orthogonal space. 
For any subset  $X\subseteq V$, put 
\begin{equation}
\label{eq:sigma-gp}
\Sigma_X:=\la \sigma_x\colon x\in X~\mbox{is nonsingular}\ra, 
\end{equation}
the group generated by the symmetries determined by nonsingular points in $X$.

%%%
\begin{prop}
\label{prop:intersection}
Let $V_{\phi}$ be an orthogonal space, and 
$W_1,W_2$ orthogonal subspaces of dimension $n\geq 3$
such that $U:=W_1\cap W_2$ is an orthogonal subspace of dimension $n-1$.
Then 
\[
\Sigma_{W_1}\cap \Sigma_{W_2}=\Sigma_{U}.
\]
\end{prop}

\begin{proof}
Put $G:=\isom(V_{\phi})$ and, for $i=1,2$, put $H_i:=\Sigma_{W_i}$. Note, $H_i\leq{\rm stab}_{G}(W_i)$, 
so $H_1\cap H_2\leq{\rm stab}_{H_1}(U)\cap{\rm stab}_{H_2}(U)$. Evidently $\Sigma_U\leq H_1\cap H_2$, 
so we must establish the reverse inclusion. Fix $i\in\{1,2\}$, and consider the stabilizer $J_i:={\rm stab}_{H_i}(U)$.
\smallskip

First, suppose $n$ is even, so that $\rad(W_i)=0$, $\rad(U)=\la z\ra$ for some $z\in U$, and
$J_i={\rm stab}_{H_i}(\la z\ra)$.
If $z$ is nonsingular, $J_i=\la \sigma_z\ra\times\isom(U)=\Sigma_U$, so
\[
H_1\cap H_2= J_1\cap J_2=\la \sigma_z\ra\times\isom(U)=\Sigma_U.
\]
Thus, we may assume $z$ is singular. Fix a nondegenerate $(n-2)$-space $U_0\leq U$ not containing $z$,
and choose a singular vector $e_i\in (W_i\cap U_0^{\perp})-U$ such that $(z,e_i)=1$.
Then $J_i$ factorizes as $J_i=Q\rtimes {\rm stab}_{J_i}(\la e_i\ra)$, where $Q=O_2(J_i)$ is the group
inducing the identity on $U/\la z\ra$. Furthermore, ${\rm stab}_{J_i}(\la e_i\ra)=\la h_i\ra\times K$,
where $h_i$ maps $z\mapsto \alpha z$, $e_i\mapsto\alpha^{-1}e_i$ and is the identity on $U_0$, 
and $K\cong\isom(U_0)$ is the subgroup of $J_i$ inducing the identity on $\la z,e_i\ra$.
We have omitted the subscripts on $Q$ and $K$ since these groups are common to both $J_1$ and $J_2$,
while $h_i\not\in J_{3-i}$. Indeed, $J_1\cap J_2=Q\rtimes K$, so it suffices to show that $Q\rtimes K\leq\Sigma_U$.

Since $U_0$ is nondegenerate, $K=\Sigma_{U_0}$. Note, if $x\in U- U_0$ nonsingular, then
$\sigma_x\in\Sigma_U$ does not normalize $K$. As $Q$ acts transitively on its set of complements in $Q\rtimes K$,
there exists $1\neq b\in Q$ such that $K^b=K^{\sigma_x}$, so
$b\sigma_x$ normalizes $K$. Since $N_{Q\rtimes K}(K)=K$, it follows that $b\sigma_x\in K$,
so $b\in\la K,\sigma_x\ra\leq\Sigma_U$. Finally, $Q$ is isomorphic to the natural orthogonal $K$-module
$U_0$, so $Q\rtimes K=\la K,b\ra\leq \Sigma_U$.
\smallskip

Next, suppose $n$ is odd. Then $W_i$ has a 1-dimensional radical $\la z_i\ra$ that can
be singular or nonsingular, and $U$ is nondegenerate. If $z_i$ is nonsingular, then 
$H_i=\la\sigma_{z_i}\ra\times\isom(W_i)$
and $J_i=\la \sigma_{z_i} \ra\times \Sigma_U$, where $\Sigma_U\cong\isom(U)$. 
If $z_i$ is singular then, as in the singular case above,  
$H_i=O_2(H_i)\rtimes {\rm stab}_{H_i}(U)$,
so that $J_i= {\rm stab}_{H_i}(U)=\Sigma_U\cong\isom(U)$. There are now three cases to consider:
$z_1$ and $z_2$ are both nonsingular, both singular, or one is singular and one is nonsingular.
Observe, however, if $z_i$ is nonsingular then $\sigma_{z_i}\not\in J_{3-i}$; thus, 
it is clear from the structure of $J_i$ that $J_1\cap J_2=\Sigma_U$ in all three cases.
\end{proof}

We will also make use of the following elementary result on generation by symmetries
in orthogonal 3-spaces. 

%%%
\begin{lem} 
\label{lem:plane}
If $q\geq 4$, $V_{\phi}$ is a $3$-dimensional orthogonal space over $\F$, and $L$ is an asingular line
of $V$,  then $\Sigma_L$ is maximal in $\Sigma_V$.
\end{lem}

\begin{proof}
If $V^{\perp}$ is nonsingular, then
$\Sigma_V=\isom(V_{\phi})\cong\Sp(V/V^{\perp})\cong\SL(2,\F)$ and
$\Sigma_L=\Orth^-(L)\cong D_{2(q+1)}$. It is well known that all
$D_{2(q+1)}$ subgroups of  $\SL(2,\F)$ are maximal when $q\geq 4$.
If $V^{\perp}$ is singular, we have
$\isom(V_{\phi})\cong Q\rtimes(\F^{\times}\times\Orth^-(L))$, where $Q=O_2(\isom(V_{\phi}))$. Thus,
as in the ``$n$ even, $z$ singular" case in the proof of 
Proposition~\ref{prop:intersection}, we have $\Sigma_V\cong Q\rtimes\Orth^-(L)\cong Q\rtimes\Sigma_L$.
In particular, $\Sigma_L$ is its own normalizer in $\Sigma_V$ and the regularity of $Q$ on its
set of complements in $\Sigma_V$ ensures that $\la g,\Sigma_L\ra=\la u, \Sigma_L\ra$  for any
$g\in\Sigma_V\setminus \Sigma_L$. Again, $Q$ is the natural
module for $\Sigma_L\cong\Orth^-(L)$ under the conjugation action, so $\la u,\Sigma_L\ra=Q\rtimes \Sigma_L=\Sigma_V$.
\end{proof}

\begin{remark}
\label{rem:no2}
{\rm
If $q=2$, $V^{\perp}$ is nonsingular, and $L$ is asingular, then $\Sigma_L\cong D_6\cong \SL(2,\Bbb{F}_2)\cong \Sigma_V$;
the condition $q\geq 4$ is needed. 
For $q\geq 8$, the result also holds for a hyperbolic line $L$ by the same argument;
for $q=4$, however, there exist elements $u\in Q$ such that $\la u,\Sigma_L\ra$ 
is a proper intermediate subgroup of $\Sigma_L$ and $\Sigma_V$.
}
\end{remark}

%%%%%
\section{string groups generated by symmetries}
\label{sec:string}
Let $V_{\phi}$ be an orthogonal space of dimension $d$ over the field $\F$, where $q=2^k$. Suppose that
$(G;\{\sigma_{v_1},\ldots,\sigma_{v_d}\})$ is a string subgroup of $\Orth(V_{\phi})$
generated by $d$ symmetries.
Considering commutator relations among symmetries, we see that $(v_i,v_j)=0$ if, and only if,
$|i-j|\neq 1$. As the $v_i$ are nonsingular vectors and our field has characteristic 2, we can
normalize the $v_i$ so that $\phi(v_i)=1$. Thus, relative to the basis $v_1,\ldots,v_d$,
the quadratic form $\phi$ has matrix
\begin{equation}
\label{eq:quad-form}
\Phi(\alpha_1,\ldots,\alpha_{d-1})~=~\left[ \begin{array}{ccccl}
                          1 & \alpha_1 &  &  &    \\
                           &   1  & \alpha_2 &  &   \\
                          &  & \ddots & \ddots &   \\
                            &  &  & 1 & \alpha_{d-1} \\
                             &  &  &  & ~~1
                 \end{array} \right],
\end{equation}
for some $\alpha_1,\ldots,\alpha_{d-1}\in\F$, where 
all other entries are 0.
The bilinear form $(\,,\,)$ associated to $\phi$ is represented
by $B=\Phi+\Phi^{{\rm tr}}$. If in addition $G$ has no direct product decomposition, then
the scalars $\alpha_i$ belong to $\F^{\times}$. 

Next, consider the matrices representing 
the symmetries $\sigma_{v_i}$ for $i=1,\ldots,d$
relative to the same basis. First, $\sigma_{v_1}$ maps $v_2\mapsto v_2+\alpha_1v_1$
and fixes the remaining basis vectors. It therefore  has matrix 
$\left[ \begin{smallmatrix} 1 & \cdot \\ \alpha_1 & 1 \end{smallmatrix} \right]
\oplus I_{d-2}$. Similarly, $\sigma_d$ has matrix 
$I_{d-2}\oplus \left[ \begin{smallmatrix} 1 & \alpha_{d-1} \\ \cdot & 1 \end{smallmatrix} \right]$.
For $2\leq i\leq d-1$,  $\sigma_i$ is represented by a matrix whose restriction to 
$\la v_{i-1},v_i,v_{i+1} \ra$ has the form
\[
\left[ \begin{array}{ccc} 1 & \alpha_{i-1} & \cdot \\ \cdot & 1 & \cdot \\ \cdot & \alpha_i & 1 \end{array} \right],
\]
inducing the identity on the span of the remaining basis vectors. 

Conversely, any choice
of scalars $\alpha_1,\ldots,\alpha_{d-1}\in\Fk^{\times}$ determines both a quadratic form
$\phi$ on $V=\F^{\,d}$ represented by the matrix $\Phi(\alpha_1,\ldots,\alpha_{d-1})$, and a string group
$(G;\{\rho_{1},\ldots,\rho_{d}\})$, where $\rho_i$ is the symmetry determined by the $i^{{\rm th}}$
(nonsingular) standard basis vector of $V$. The properties of both $V_\phi$ and the string group
$(G;\{\rho_{1},\ldots,\rho_{d}\})$ are determined by the scalars $\alpha_i$. 

In order to prove Theorem~\ref{thm:main-orth} we shall need to demonstrate that the scalars
$\alpha_1,\ldots,\alpha_{d-1}$ may be selected so that
$(G;\{\rho_{1},\ldots,\rho_{d}\})$ also satisfies the intersection property 
needed to make it a string C-group. We take up this issue in the next section.
To prove Corollary~\ref{coro:main-sp}, we must also understand the radical of an orthogonal space 
$V_{\phi}$ having matrix $\Phi(\alpha_1,\ldots,\alpha_{d-1})$ and, for nondegenerate $V_{\phi}$
in even dimension, how the choice of scalars affects the Witt index of the space.
The following characterization of the radical of $V_{\phi}$ solves the first of these problems.  

%%%
\begin{lem}
\label{lem:nondeg}
If $V=\F^{\,d}$,
$\alpha_1,\ldots,\alpha_{d-1}$ are any elements of $\Fk^{\times}$, and $\phi$ is the quadratic form
represented by $\Phi=\Phi(\alpha_1,\ldots,\alpha_{d-1})$, then $V_{\phi}$ is an orthogonal space.
In particular, if $d$ is even then $V_{\phi}$ is nondegenerate, and if $d=2m+1$ is odd then $V_{\phi}$
has 1-dimensional radical spanned by 
\begin{align}
\label{eq:radical}
z=( 1,0,\beta_1,0,\ldots,0,\beta_m), & & where~
\beta_s=\prod_{i=1}^s\left( \frac{\alpha_{2i-1}}{\alpha_{2i}} \right)~\mbox{for}~s=1,\ldots,m.
\end{align}
Furthermore, the radical is singular precisely when $1+\beta_1^{\,2}+\ldots+\beta_m^{\,2}=0$.
\end{lem}

\begin{proof}
This is a direct calculation. Recall that the matrix $B=\Phi+\Phi^{{\rm tr}}$ represents the symmetric form 
determined by $\Phi$, and $\la z\ra$  is its nullspace. Thus, whether $V^{\perp}$ is 
singular or nonsingular is determined by whether $\phi(z)=z\Phi z^{{\rm tr}}$ is zero or nonzero, respectively;
this gives rise to the condition in the lemma.
\end{proof}

Suppose that $d$ is even, and let $\Phi=\Phi(\alpha_1,\ldots,\alpha_{d-1})$ for some $\alpha_i\in\F^{\times}$.
Then, by Lemma~\ref{lem:nondeg}, the orthogonal space $V_{\phi}$ is nondegenerate. We now determine
the Witt index of $V_{\phi}$. Let $e_1,f_1,\ldots,e_m,f_m$ be any hyperbolic basis of the symplectic space
$V$ associated with $V_{\phi}$. Let $N$ denote the additive subgroup of $\F^+$ consisting of all elements $\alpha^2+\alpha$,
where $\alpha\in\F$. The {\em Arf invariant} of $V_{\phi}$ is
\begin{align}
\label{eq:Arf}
{\rm Arf}(V_{\phi}) &= \sum_{i=1}^m\phi(e_i)\phi(f_i)~~~{\rm mod}~N,
\end{align}
an element of $\F^+/N\cong \Bbb{Z}_2$. The Witt index of $V_{\phi}$ is $m$ or $m-1$ according as
${\rm Arf}(V_{\phi})$ is 0 or 1 mod $N$, respectively~\cite{Arf}. Thus, we next compute a hyperbolic basis for $(\,,\,)$ by
applying the Gram-Schmidt process to $B=\Phi+\Phi^{{\rm tr}}$.

For successive $i=0,\ldots,1+(d-4)/2$, add $\alpha_{2i+2}/\alpha_{2i+1}$ times row-column $2i+1$ to row-column $2i+3$. 
This process reduces $B$ to the matrix 
\[
\left[ \begin{array}{cc} 0 & \alpha_1 \\ \alpha_1 & 0 \end{array} \right]
\oplus
\left[ \begin{array}{cc} 0 & \alpha_3 \\ \alpha_3 & 0 \end{array} \right]
\oplus \ldots \oplus
\left[ \begin{array}{cc} 0 & \alpha_{d-1} \\ \alpha_{d-1} & 0 \end{array} \right]
,
\]
which is almost a hyperbolic basis. Indeed, scaling alternate vectors in the 
resulting basis, we obtain the desired hyperbolic basis
$e_1,f_1,\ldots,e_m,f_m$, where $e_1=v_1$, 
\begin{align}
\label{eq:ei}
e_i &= v_{2i-1}+\sum_{j=1}^{i-1} \left( \prod_{\ell=1}^j\frac{\alpha_{2i-2\ell}}{\alpha_{2i-2\ell-1}}\right)v_{2i-2j-1} & \mbox{for}\;2\leq i\leq m,
\end{align}
and $f_i=v_{2i}/\alpha_{2i-1}$ for $1\leq i\leq m$. Although this is a rather unwieldy formula, we can
simplify it greatly through judicious choices of scalars. We shall return to this when we prove Corollary~\ref{coro:main-sp}
in the following section.

%%%%%
\section{proofs of the main results}
\label{sec:proof}
We saw in the previous section that any selection of scalars $\alpha_1,\ldots,\alpha_{d-1}$ from $\F^{\times}$ defines
a quadratic form $\phi$ represented on the row space $V=\F^d$ by $\Phi(\alpha_1,\ldots,\alpha_{d-1})$, turning
$V_{\phi}$ into an orthogonal space. If $v_1,\ldots,v_d$ denotes the standard basis of $V$, we also saw
that $G=\la \sigma_{v_1},\ldots,\sigma_{v_d}\ra$ is a  string group. We now show that
the scalars $\alpha_i$ may be chosen so that $G=\Orth(V_{\phi})$ and $(G;\{ \sigma_{v_1},\ldots,\sigma_{v_d}\})$
is a string C-group, thereby proving Theorem~\ref{thm:main-orth}.
\medskip

Fix $1\leq i\leq d-1$, and consider the dihedral group $\la \sigma_{v_i},\sigma_{v_{i+1}} \ra$, which
acts on its 2-dimensional support $\la v_i,v_{i+1}\ra$. We select the scalar $\alpha_i$
so that two conditions are satisfied: first that $\la v_i,v_{i+1}\ra$ is asingular,
and secondly that $\la \sigma_{v_i},\sigma_{v_{i+1}}\ra$ induces the full group 
$\Orth^-(\la v_i,v_{i+1}\ra)$, of order $2(q+1)$, on this 2-space. 
The restriction of $\phi$ to
$\la v_i,v_{i+1}\ra$ is represented by the matrix $\left[ \begin{smallmatrix} 1 & \alpha_i \\ 0 & 1 \end{smallmatrix} \right]$;
the Arf invariant of this  form is $\alpha_i^{-1}\;({\rm mod}\;N)$. Thus, the first condition is satisfied
so long as all  $\alpha_i$ are chosen from the set
\begin{align}
\label{eq:define-A0}
A_0 &= \{ \beta\in\F^*\colon \beta^{-1}\not\in N \},
\end{align}
of size $\frac{q}{2}$.
Next, the restriction of the product
$h_i:=\sigma_{v_i}\sigma_{v_{i+1}}$ has matrix 
\begin{equation}
\label{eq:def-ai}
H_{\alpha_i}:=\left[ \begin{array}{cc} 1 & \alpha_{i} \\ \alpha_i & 1+\alpha_i^2 \end{array} \right]\in\SL(2,\F).
\end{equation} 
For $\alpha_i$ selected from $A_0$, the order of this matrix divides $q+1$. Thus, we wish to select 
our scalars from the set $A=\{\beta\in A_0\colon H_{\beta}\;\mbox{has order}\;q+1\}$. The proportion
of elements of $A_0$ that belong to $A$ is determined by the Euler totient of $q+1$.
\medskip

The following somewhat technical result
is crucial to our proof of Theorem~\ref{thm:main-orth}.

%%%
\begin{lem}
\label{lem:generate}
Let $q=2^k\geq 4$ and $V$ a vector space of dimension $d$ over $\F$.
Let $\alpha_1,\ldots,\alpha_{d-1}\in A$, and let 
$\phi$ be the quadratic form represented in a basis $v_1,\ldots,v_d$ of $V$ 
by the matrix $\Phi(\alpha_1,\ldots,\alpha_{d-1})$ in equation~(\ref{eq:quad-form}).
For each $1\leq b\leq d-1$, and each $1\leq i\leq d-b+1$, if 
\begin{align*}
X_{b,i}=\{v_i,\ldots,v_{i+(b-1)}\},~~ \mbox{then} & ~~~~~~
\Sigma_{X_{b,i}}=
\Sigma_{\la X_{b,i} \ra}.
\end{align*}
In words, the group generated by the symmetries determined by consecutive basis vectors always coincides with
the (potentially larger) group associated with the linear span of these vectors.
\end{lem}

\begin{proof}
We proceed by induction on $b$, with the case $b=1$ being obvious.  
\smallskip

When $b=2$, we are given nonsingular  $v_i,v_{i+1}$
such that $\la v_i,v_{i+1}\ra$ is asingular and $h := \sigma_{v_i}\sigma_{v_{i+1}}$ has order $q+1$, and 
we wish to show that $\Sigma_{\{v_i,v_{i+1}\}}=\Sigma_{\la v_i,v_{i+1}\ra}$.
Note, $\la h\ra$ acts regularly on the set of points of $\la v_i,v_{i+1}\ra$.
Thus, for each nonzero vector $x$ in $\la v_i,v_{i+1}\ra$, there exists $0\leq j<q+1$ such that $\la x\ra=\la v_i\ra h^j$. 
Hence, $\sigma_x=\sigma_{v_i}^{h^j}\in
\Sigma_{\{v_i,v_{i+1}\}}$, as required.
\medskip

The case $b=3$ follows from $b=2$ together with Lemma~\ref{lem:plane}.
For, if $L=\la v_i,v_{i+1}\ra$, then $\Sigma_L=\Sigma_{X_{2,i}}\leq\Sigma_{X_{3,i}}$
by the $b=2$ case. 
By Lemma~\ref{lem:plane}, $\Sigma_L$ is maximal in $\Sigma_{\la X_{3,i}\ra}$. 
As $\sigma_{v_{i+2}}\not\in\Sigma_L$, so $\Sigma_{\la X_{3,i}\ra}=\Sigma_{L\cup\{v_{i+2}\}}\leq
\Sigma_{X_{3,i}}$. (Note, this argument applies to the $\Bbb{F}_2$ case as well, since $\phi$ has a singular
radical.)
\smallskip

We can use Lemma~\ref{lem:plane} to treat the inductive $b>2$ case uniformly
as follows.
Consider $y\in \la X_{b,i}\ra$ nonsingular;
we must show that $\sigma_y\in \Sigma_{X_{b,i}}$. Let $W= \la X_{b-1,i+1}\ra$.
By induction, $\Sigma_W=\Sigma_{X_{b-1,i+1}}\leq\Sigma_{X_{b,i}}$, so
 we may assume that $\la y\ra$ lies off $W$.
Similarly, if  $L=\la v_i,v_{i+1}\ra$, then $\Sigma_L=\Sigma_{X_{2,i}}\leq \Sigma_{X_{b,i}}$ by the $b=2$ case,
so we may further assume that $\la y\ra $ lies off $L$.
Thus, the plane $U=\la v_i,v_{i+1},y\ra$ meets $W$ in a line, $M$,
containing $\la v_{i+1}\ra$. 
As $M$ is not totally singular and contains $q\geq 4$ points not equal to $\la v_{i+1}\ra$,
at least one of those, $x$, is nonsingular. 
Evidently, $\sigma_x\in\Sigma_W\in\Sigma_{X_{b,i}}$. 
Now, by the argument in the foregoing paragraph,
it follows that $\Sigma_U=\Sigma_{L\cup\{x\}} \leq  \Sigma_{X_{b,i}}$.
In particular $\sigma_y\in
\Sigma_{X_{b,i}}$, as claimed.
\end{proof}

\begin{remark}
\label{rem:3}
{\rm We might equally have chosen the $\alpha_i$ so that the matrices in~(\ref{eq:def-ai}) have order $q-1$
or any combination of $q-1$ and $q+1$. The crucial property is that $\la \sigma_{v_i},\sigma_{v_{i+1}}\ra$
induces the full group of isometries on its support $\la v_i,v_{i+1}\ra$. There is one exception: when $q=4$, 
strings $\sigma_{v_{i}},\ldots,\sigma_{v_{i+t}}$ such that $\sigma_{v_j}\sigma_{v_{j+1}}$ has order $q-1=3$
generate groups $S_{t+1}$ rather than $\Orth(\la v_i,\ldots,v_{i+t}\ra)$; cf.~Remark~\ref{rem:no2}. 
We insisted on using elements of order
 $q+1$ to avoid dealing with this exception.
}
\end{remark}
\medskip

We shall need to check that certain sequences of
involutions satisfy the intersection property~(\ref{eq:IP}). This would be quite tedious
if one had to check all possible intersections, but the following result makes the task tractable.

\begin{prop}{\cite[Proposition 2E16]{MS2002}}
\label{prop:IP}
Let $G$ be a group generated by involutions $\rho_1,\ldots,\rho_{n}$. If both 
 $\la \rho_1,\ldots,\rho_{n-1}\ra$ and $\la \rho_2,\ldots,\rho_n\ra$ are string 
 C-groups on their defining generators, and 
\[
\la \rho_1,\ldots,\rho_{n-1}\ra \cap
\la \rho_2,\ldots,\rho_n\ra =
\la \rho_2,\ldots,\rho_{n-1}\ra, 
\] 
then $\rho_1,\ldots,\rho_n$ satisfies the
intersection property. In particular, $(G;\{\rho_1,\ldots,\rho_n\})$ is a string C-group.
\end{prop}
\medskip

\noindent {\bf Proof of Theorem~\ref{thm:main-orth}.}~
Let $q=2^k$ for $k\geq 2$, and $V=\F^d$ for $d\geq 2$.
To begin with, choose any scalars $\alpha_1,\ldots,\alpha_{d-1}\in A$.
Let $\phi$ be the quadratic form on $V$ represented relative to a basis $v_1,\ldots,v_d$
by the matrix $\Phi(\alpha_1,\ldots,\alpha_{d-1})$ in equation~(\ref{eq:quad-form}).
We show that $(\Sigma_V;\{\sigma_{v_1},\ldots,\sigma_{v_d}\})$ is a string C-group.
The sequence $\sigma_{v_1},\ldots,\sigma_{v_d}$ 
satisfies the necessary string condition by construction and they generate 
$\Sigma_V$ by Lemma~\ref{lem:generate}. It
remains only to establish the intersection property. 

We proceed by induction on $d$, following the notation of Lemma~\ref{lem:generate}.
Recall, $\la \sigma_{v_1},\sigma_{v_2}\ra$ is the dihedral group $D_{2(q+1)}\cong\Orth^-(2,\F)$
so the result is clear for the base case $d=2$.
Let $d>2$ and suppose
that $\Sigma_{X_{b,i}}$ (with its defining generators) satisfies the intersection property whenever $b<d$.
In particular, both $\Sigma_{X_{d-1,1}}$ and $\Sigma_{X_{d-1,2}}$ satisfy the intersection property.
Note, $\la X_{d-1,1}\ra\cap\la X_{d-1,2}\ra=\la X_{d-2,2}\ra$, so it follows from 
Lemma~\ref{lem:generate} and Proposition~\ref{prop:intersection} that
\[
\Sigma_{X_{d-1,1}}\cap\Sigma_{X_{d-1,2}}=\Sigma_{\la X_{d-1,1}\ra}\cap\Sigma_{\la X_{d-1,2}\ra}
=\Sigma_{\la X_{d-2,2} \ra}=\Sigma_{X_{d-2,2}}.
\]
The intersection property now follows from Proposition~\ref{prop:IP}. 

When $q=2$, $A=\{1\}$ so there is only one choice for each of the scalars $\alpha_i$.
Here, as in~Remark~\ref{rem:3},
each $\sigma_{v_i}\sigma_{v_{i+1}}$ has order 3, so $\la\sigma_{v_1},\ldots,\sigma_{v_d}\ra\cong S_d$. 
Thus, it is not possible to generate $\Orth(d,\Bbb{F}_2)$ as a string C-group using symmetries.
\hfill $\Box$
\bigskip

\noindent {\bf Proof of Corollary~\ref{coro:main-sp}.}~ First, let $d=2m$. We must show that,
given either possible isometry type $\epsilon\in\{+,-\}$ of a quadratic space $V=\F^d$,
the scalars $\alpha_1,\ldots,\alpha_{d-1}\in A$ may be chosen to endow $V$
with a form $\Phi(\alpha_1,\ldots,\alpha_{d-1})$ of type $\epsilon$. To this end, we employ
the Arf invariant described in the preceding section.

For $\lambda,\mu\in A$, put $\alpha_1=\alpha_{d-1}=\lambda$ and $\alpha_i=\mu$ for $1<i<d-1$.
Using equations~(\ref{eq:Arf}) and~(\ref{eq:ei}), we compute
\begin{equation}
\label{eq:prescribed-arf}
{\rm Arf}(V_{\phi}) = \frac{m(m-1)}{2}\cdot \mu^{-1}+(m-1)\cdot \lambda^{-1}+\frac{\mu}{\lambda^2}~~~~({\rm mod}\;N).
\end{equation}
If $\mu=\lambda$, then~(\ref{eq:prescribed-arf}) reduces to ${\rm Arf}(V_{\phi})=\frac{m(m+1)}{2}\cdot \lambda^{-1}\;({\rm mod}\;N)$.
As $\lambda\in A$, so $\lambda^{-1}\not\in N$; hence, in this case ${\rm Arf}(V_{\phi})=0$ if, and only if, $m$ is congruent
to 0 or 3 mod 4. If, on the other hand, we choose $\mu$ so that $\frac{\mu}{\lambda^2}\in N$, then
${\rm Arf}(V_{\phi})=0$ if, and only if, $m$ is congruent to 1 or 2 mod 4. Hence, whatever the value of $m$, 
we can construct a suitable quadratic form on $\Fk^{2m}$ of either isometry type. In particular,
$\Orth^{\pm}(2m,\Fk)$ is always a string C-group of rank $2m$, as claimed.
\smallskip

Next, take $d=2m+1$ to be odd. Referring to Lemma~\ref{lem:nondeg}, 
the 1-dimensional radical $V^{\perp}$ is nonsingular
precisely when $1+\beta_1^{\,2}+\ldots+\beta_m^{\,2}\neq 0$.
Observe that the value of $\beta_m^{\,2}$ is changed if one fixes all $\alpha_i$ for $1\leq i<d-1$ 
and makes a different choice for $\alpha_{2m}=\alpha_{d-1}$. Hence, we can always arrange
for $V^{\perp}$ to be nonsingular. Having done so, we have 
$\Sp(2m,\F)\cong\Sp(V/V^{\perp})\cong\isom(V)$, and the result follows. \hfill $\Box$

%%%%%
\section{low-dimensional classical groups acting on polytopes}
\label{sec:conclusion}
There has been a significant effort in recent years to determine the pairs $(G,r)$ such that the finite simple 
group $G$ is a string C-group of rank $r$. 
The question has been settled completely for $r=3$ by the combined results of Nuzhin for simple groups
of Lie type~\citelist{\cite{Nu1}\cite{Nu2}} and Mazurov for sporadic simple groups~\cite{Maz}.

For higher ranks, the question has also been settled for alternating groups in three separate 
papers involving the third author~\citelist{\cite{FLM}\cite{FLM2}\cite{CFLM}}. For other types of simple
groups the study has generally progressed in a more piecemeal fashion by considering
individual infinite families such as $\PSL(2,\F)$, $\PSL(3,\F)$, and $\PSL(4,\F)$ 
(see~\citelist{\cite{ls07}\cite{BV10}\cite{BL:psl4}}, respectively). The results in this paper,
as well as contributing generally to the ongoing effort to understand abstract regular polytopes
of higher rank, help to settle the string C-group question for low-dimensional simple classical groups
defined over finite fields of characteristic 2. 

%%%%%
\begin{thm}
\label{thm:low}
Let $\F$ be the finite field with $q=2^k$ elements, $V$ an $\F$-vector space of dimension at most 4,
and $G$ a quasisimple classical group whose natural module is $V$. Then the associated simple group
$G/Z(G)$ is a string C-group of rank $r\geq 4$ if, and only if, $r\in \{4,5\}$ and $G=\Sp(4,\F)$.
\end{thm}

\begin{proof}
We first show that no quasisimple classical group $G$ other than $\Sp(4,\F)$ is such that $G/Z(G)$
is a string C-group of rank $r\geq 4$. First, Leemans and Schulte~\cite{ls07} showed that
$\PSL(2,\F)$ is a string C-group of rank $r\geq 4$ if, and only if, $r=4$ and $q\in \{11,19\}$.
Secondly, Brooksbank and Vicinsky~\cite{BV10} showed that if $G$ is an irreducible string 
subgroup of $\GL(3,\F)$, then $G$
preserves a symmetric bilinear form, so $G/Z(G)\cong\PSL(2,\F)$.
Finally, Brooksbank and Leemans~\cite[Corollary 4.5]{BL:psl4} showed that $\PSL(4,\Fk)$ is not a string C-group
of any rank. The latter was done by showing that the isomorphic group $\Omega^+(6,\Fk)$ is not a string C-group of any rank,
and the exact same proof applies to $\Omega^-(6,\Fk)\cong{\rm PSU}(4,\Fk)$.
\smallskip

Next, suppose that $G=\Sp(4,\Fk)$; note $Z(G)=1$, so $G$ is simple. 
We may assume that $k\geq 2$, since $\Sp(4,\Bbb{F}_2)\cong{\rm Sym}(6)$ is not quasisimple.
(Moreover, the derived group $\Sp(4,\Bbb{F}_2)'\cong{\rm Alt}(6)$ is not a string C-group of rank $r\geq 4$.)
By Corollary~\ref{coro:main-sp}, $G$ is a string C-group of rank 5. 
However, $G$ is not a string C-group of rank $r\geq 6$ since none of its maximal subgroups
contain a string C-group of rank $r-1\geq 5$; see~\cite[Table 8.14]{BHR}.
It therefore remains to show that $G$ is also a string C-group of rank 4,
which we do by constructing it as such.
\smallskip

As usual, let $(\,,\,)$ be a nondegenerate symmetric bilinear form on $V$, and $\phi$ a nondegenerate
quadratic form of Witt index 1 whose associated symmetric form is $(\,,\,)$. Let $L=\la e,b\ra$
be a 2-dimensional subspace of the orthogonal space $V_{\phi}$ with $\phi(e)=(e,b)=0$ and $\phi(b)\neq 0$.
Fix $0\neq \alpha\in\F$ and let $\rho$ be the {\em pseudo-transvection}
\[
\rho\colon v\mapsto v+\alpha[(v,b)e+(v,e)b]+\alpha^2\phi(b)(v,e)e.
\]
Consider nonsingular vectors $u,w\not\in\la e\ra^{\perp}$ with $(u,w)=0$; we will refine the
selection of $u,w$ subject to these fundamental conditions.
Note, $U=\la u,L\ra$ is a 3-space with nonsingular radical $\la x\ra=\la u\ra^{\perp}\cap L$.
Similarly, $W=\la w,L\ra$ has nonsingular radical $\la y\ra=\la w\ra^{\perp}\cap L$.

Let us consider $U$. Note that $\rho$ induces the identity on $L$, but moves $\la u\ra$.
Hence, $\la \sigma_u,\rho\ra$ fixes both the radical $\la x\ra$ of $U$ and the line $\la u,u\rho\ra$.
By a suitable choice of $u$, one can arrange for $\la u,u\rho\ra$ to be asingular, and for the
group induced on $\la u,u\rho\ra$ by $\la \sigma_u,\rho\ra$ to be the full group $\Orth^-(\la u,u\rho\ra)\cong D_{2(q+1)}$
of isometries. Indeed, for such $u$ we have 
$\la \sigma_u,\rho\ra\cong\Orth^-(\la u,u\rho\ra)\times\la\sigma_x\ra \cong D_{4(q+1)}$. 

In the same way, we can select $w$ so that $\la \sigma_w,\rho\ra\cong \Orth^-(\la w,w\rho\ra)\times\la\sigma_y\ra
\cong D_{4(q+1)}$. 
Note that the condition $(u,w)=0$ ensures that $[\sigma_u,\sigma_w]=1$, so $\la \sigma_u,\rho,\sigma_w\ra$
is a string group. Furthermore, we have $\rho\in\Orth(V_{\phi})'=\Omega(V_{\phi})$, while
$\sigma_u,\sigma_w\in\Orth(V_{\phi})\setminus\Omega(V_{\phi})$. Hence, $(\sigma_u\rho)^2$ and $(\sigma_w\rho)^2$
are non-commuting elements of order $q+1$ in $\Omega^-(4,\F)\cong\SL(2,\Bbb{F}_{q^2})$, and so generate this group.
It follows that  $\la \sigma_u,\rho,\sigma_w\ra\cong\Orth^-(V_{\phi})$.

Consider the group $T$ of transvections inducing the identity on $\la x\ra^{\perp}=U$ and on $V/\la x\ra$. 
Note that $\sigma_x\in T\cong \F^+$
and is the only element of this group belonging to $\Orth(V_{\phi})$. As $q\geq 4$, 
there exists $\tau\in T\setminus\{1,\sigma_x\}$. 
The choice of $\tau$ ensures that $\la \sigma_u \rho,\sigma_w,\tau\ra$ is a string group.
Further, since $\Orth^-(V_{\phi})$ is maximal in $\Sp(V)$ and $\tau\not\in \Orth^-(V_{\phi})$,
we have $\Sp(V)=\la \sigma_u, \rho,\sigma_w,\tau\ra$.

Next, observe that $\la \rho,\sigma_w,\tau\ra$ fixes the plane $W$, and hence
$W^{\perp}=\la y\ra$. Indeed, $\la \rho,\sigma_w,\tau\ra$ preserves a quadratic form $\phi'$ on $W$ distinct from the restriction
of $\phi$, and the structure of the group is determined by whether or not $\la y\ra$ is nonsingular with respect to this 
form, as described in the proof of Lemma~\ref{lem:plane}. 
For us, what matters in each case is the intersection
of that group with $\la \sigma_u,\rho,\sigma_w\ra$. To that end, we consider the stabilizer in $\la \rho,\sigma_w,\tau\ra$
of the line $\la w,w\rho\ra$. If $\phi'(y)\neq 0$, then the stabilizer is $\la \rho,\sigma_w\ra\times\la \sigma_y'\ra$,
where $\sigma_y'$ is the symmetry determined by $y$ relative to $\phi'$. This is an element of the group of transvections
inducing the identity on $\la y\ra^{\perp}$ and $V/\la y\ra$, but is distinct from $\sigma_y\in \la \rho,\sigma_w\ra$, so
$\sigma_y'\not\in \la \sigma_u,\rho,\sigma_w\ra$. If $\phi'(y)=0$, on the other hand, then the stabilizer of 
$\la w,w\rho\ra$ is simply $\la \rho,\sigma_w\ra$. Thus, in each case
\[
\la \sigma_u,\rho,\sigma_w\ra \cap \la \rho,\sigma_w,\tau\ra = \la \rho,\sigma_w\ra,
\]
so $(\Sp(V);\{ \sigma_u, \rho,\sigma_w,\tau \})$ satisfies the intersection
property by Proposition~\ref{prop:IP} and is therefore a string C-group, as required.
\end{proof}

For completeness, we describe suitable involutions $\sigma_u, \rho,\sigma_w,\tau$ 
explicitly as matrices relative to the ordered basis $u,e,b,w$. Let $\alpha\in\F^{\times}$ be such that
$\left[ \begin{smallmatrix} 0 & 1 \\ 1 & \alpha^2 \end{smallmatrix} \right]$ has order $q+1$.
Now, put
\[
\sigma_u:=\left[ \begin{smallmatrix} 1 & 0 & 0 & 0 \\ \alpha & 1 & 0 & 0 \\  0 & 0 & 1 & 0 \\ 0 & 0 & 0 & 1 \end{smallmatrix} \right],~~~
\rho:=\left[ \begin{smallmatrix} 1 & \alpha & \alpha & 0 \\ 0 & 1 & 0 & 0 \\ 0 & 0 & 1 & 0 \\ 0 & 1 + \alpha & \alpha & 1 \end{smallmatrix} \right],~~~
\sigma_w:=\left[ \begin{smallmatrix} 1 & 0 & 0 & 0 \\  0 & 1 & 0 & \alpha \\ 0 & 0 & 1 & 1 \\ 0 & 0 & 0 & 1 \end{smallmatrix} \right],~~~
\tau:=\left[ \begin{smallmatrix} 1 & 0 & 0 & 0 \\  0 & 1 & 0 & 0 \\ 0 & 0 & 1 & 0 \\ 0 & 0 & \alpha & 1 \end{smallmatrix} \right].
\]

%%%%%
\section{Concluding remarks}
To the reader interested in computer experimentation with abstract regular polytopes, we mention that
{\sc magma} code is available upon request from the authors to construct the generators exhibiting $\Orth^{\pm}(d,\Fk)$
as string C-groups of rank $d$, and $\Sp(d,\Fk)$ as a string C-group of rank $d+1$. 

As noted at the start of the preceding section, a problem of general interest is to determine the pairs $(G,r)$ such that
the finite simple group $G$ is a string C-group of rank $r$. The main results in this paper provide partial answers 
for the simple groups $\Sp(d,\Fk)$ and the almost simple groups $\Orth^{\pm}(d,\Fk)$, in particular showing that we 
can build abstract regular polytopes of arbitrarily high rank whose group of automorphisms is an orthogonal or 
symplectic group. In proving
Theorem~\ref{thm:low}, we further showed that as well as having polytopes of rank 5 the simple groups $\Sp(4,\Fk)$
have polytopes of rank 4. 

Many open questions remain. First, are the orthogonal groups $\Orth(d,\Bbb{F}_2)$ string C-groups? 
Theorem~\ref{thm:main-orth} shows they cannot be constructed as such just using symmetries, but the proof
of Theorem~\ref{thm:low} shows how other types of involutions can be exploited. (Computer 
experiments provide evidence that $\Orth(d,\Bbb{F}_2)$ is a string C-group of rank $d-1$.)
Secondly, are the simple groups $\Omega^{\pm}(2m,\Fk)$ string C-groups for $m\geq 3$? In the proof of
Theorem~\ref{thm:low} we remarked that $\Omega^{\pm}(6,\Fk)$ are {\em not} string C-groups of any rank.
The proof of this fact in~\cite{BL:psl4} relied on the very specific nature of the involutions in these groups
and the consequent geometric constraints imposed upon strings of involutions. In higher dimensions
there is a greater diversity of involution classes in $\Omega^{\pm}(2m,\Fk)$ and it is not known whether
they can be generated as string C-groups.

Perhaps the most interesting question in this line, though, is whether or not the most fundamental family of finite simple classical 
groups---the projective special linear groups---have polytopes of arbitrarily high rank. The first and third authors
showed in~\cite{BL:psl4} that $\PSL(4,\F)$ have polytopes of rank 4 when $q$ is odd, but no other constructions
are known.
\medskip

\noindent {\bf Acknowledgment.} The authors are grateful to Bill Kantor for suggesting the use of the Arf invariant,
which greatly simplified the calculations of the Witt indices of our orthogonal spaces.

%%%%%%%%%%%%%%
\bibliographystyle{amsplain}

\end{document}